\documentclass{amsart}
\newtheorem{theorem}{Theorem}[section]

\newtheorem{proposition}[theorem]{Proposition}

\newtheorem{es}[theorem]{Example}
\newtheorem{remark}[theorem]{Remark}
\def\erre{{\rm I\!R}}
\def\RR{{\rm I\!R}}

\def\essinf{\mathop {\rm ess\,inf}}

\def\phi{\varphi}

\def\esssup{\mathop{\rm esssup}}

\title[Existence and localization of
solutions...]{Existence and localization of
solutions for nonlocal fractional equations}
\author{Giovanni Molica Bisci}
\address[G. Molica Bisci]{Dipartimento P.A.U., Universit\`a  degli
Studi Mediterranea di Reggio Calabria, Salita Melissari - Feo di
Vito, 89124 Reggio Calabria, Italy} \email{gmolica@unirc.it}
\author{Du\v{s}an Repov\v{s}}
\address[D. Repov\v{s}]{Faculty of Education, and Faculty of Mathematics and Physics\\ University of Ljubljana, POB 2964, Ljubljana, Slovenia 1001}
\email{dusan.repovs@guest.arnes.si}
\thanks{Typeset by \LaTeX}
\begin{document}
\begin{abstract}
 This work is devoted to the study of the existence of at least one weak solution to nonlocal equations involving a general
integro-differential operator of fractional type. As a special case, we derive an
existence theorem for the fractional Laplacian, finding a nontrivial weak
solution of the equation
\begin{eqnarray*}
\begin{cases}
(-\Delta)^s u=h(x)f(u) & {\mbox{ in }} \Omega\\
u=0 & {\mbox{ in }} \erre^n\setminus \Omega,
\end{cases}
\end{eqnarray*}
where $h\in L^{\infty}_+(\Omega)\setminus\{0\}$ and $f:\erre\rightarrow\erre$ is a suitable continuous function.
These problems have a
variational structure and we find a nontrivial weak solution for them
by exploiting a recent local minimum result for smooth functionals defined on a reflexive Banach space.  To make the nonlinear methods work, some
careful analysis of the fractional spaces involved is necessary.
\end{abstract}
\maketitle

\section{Introduction}
The aim of this paper is to prove some existence results for fractional Laplacian problems whose prototype is
\begin{equation} \tag{$D_{h}^{f}$}\label{prototipo}
\left\{
\begin{array}{ll}
(-\Delta)^s u=h(x)f(u) & {\mbox{ in }} \Omega\\
u=0 & {\mbox{ in }} \erre^n\setminus \Omega,
\end{array} \right.
\end{equation}
where $(-\Delta)^s$ denotes the fractional Laplacian operator with $s\in (0,1)$, and $\Omega$ is an open bounded set with Lipschitz boundary of
$\RR^n$, requiring that $n>2s$. Moreover, $h\in L^{\infty}(\Omega)\setminus\{0\}$ is a nonnegative map and $f:\erre\rightarrow \erre$ represents a subcritical continuous function.\par

The existence of weak solutions for such type of problems has been intensively studied under different assumptions on the nonlinearities (see, for instance, \cite{valpal, Molica, svmountain, svlinking, servadeivaldinociBN} and references therein).\par

Moreover, the existence and multiplicity of solutions for elliptic equations in $\erre^n$, driven by a nonlocal integro-differential operator, whose standard prototype is the fractional Laplacian, have been studied, very recently, by Autuori and Pucci in \cite{AP} (this work is related to the results on general quasilinear elliptic problems given in \cite{AP0}).\par

Motivated by this wide interest, we prove in the present note some existence results (see Theorem~\ref{Esistenza} and its consequences) for fractional equations assuming that $f$ has a suitable behaviour at zero together with some global properties formulated by the means of an auxiliary function $\psi$.\par
 The strategy for proving Theorem~\ref{Esistenza} is based on the fact that our problem can be seen as the Euler--Lagrange equation of a suitable functional defined in a Sobolev space $X_0$.\par
 Hence, the solutions of \eqref{prototipo} or more generally of problem \eqref{Alf} defined in the sequel, can be found as critical points of this functional: for this purpose, along the paper, we will exploit a critical point result due to Ricceri (see Theorem~\ref{CV} and Proposition~\ref{Lemma} below).\par
 Exploiting this result, a key point is to prove the existence of a suitable $\sigma>0$ such that
 $$
\frac{\displaystyle\sup_{\|u\|_{X_0}\leq \sigma}\int_\Omega h(x)\left(\int_0^{u(x)}f(t)dt\right)\,dx}{\sigma}<\frac{1}{2}.
$$

\indent One of the main novelties here is that, in contrast with several known results (see references contained in \cite{R4}), we obtain the above inequality without continuous embedding of the ambient space in $C^0(\bar \Omega)$.\par

\indent In the nonlocal framework, denoting by $\lambda_{1,s}$ the first eigenvalue of the problem
$$ \left\{
\begin{array}{ll}
(-\Delta)^s u=\lambda u & {\mbox{ in }} \Omega\\
u=0 & {\mbox{ in }} \erre^n\setminus \Omega,
\end{array} \right.
$$

\noindent the simplest example we can deal with is
given by the fractional Laplacian, according to the following
result.\par

\begin{theorem}\label{Esistenza3}
Let $s\in (0,1)$, $n>2s$ and let $\Omega$ be an open bounded set of
$\RR^n$ with Lipschitz boundary.
Moreover, let $f:[0,+\infty)\rightarrow [0,+\infty)$ be a continuous function satisfying the following hypotheses$:$

\begin{itemize}
\item[$({\rm h}_1)$]for some $q\in\left[1,\frac{2n}{n-2s}\right)$ the function
$$
t\mapsto \frac{f(t)}{t^{q-1}}
$$
is strictly decreasing in $(0,+\infty)$ and $\displaystyle\lim_{t\rightarrow +\infty}\frac{f(t)}{t^{q-1}}=0$.
\end{itemize}
Further, suppose that
$$
\lim_{\xi\rightarrow +\infty}\frac{F(\xi)}{\xi^2}=0\,\,\,\,{and}\,\,\,\, \liminf_{\xi\rightarrow 0^+}\frac{F(\xi)}{\xi^2}>0.
$$
Then, for every
$$
\alpha>\frac{\lambda_{1,s}}{2\displaystyle\liminf_{\xi\rightarrow 0^+}\frac{F(\xi)}{\xi^2}},
$$
the parametric nonlocal problem
$$ \left\{
\begin{array}{ll}
(-\Delta)^s u=\alpha f(u) & {\mbox{ \rm in }} \Omega\\
u=0 & {\mbox{ \rm in }} \erre^n\setminus \Omega,
\end{array} \right.
$$
admits at least one non-negative and non-zero weak solution
$u_\alpha\in  H^s(\erre^n),$
such that $u_{\alpha}=0$ a.e. in $\erre^n\setminus\Omega$.
\end{theorem}
The plan of the paper is as follows; Section 2 is devoted to our abstract framework and preliminaries. Successively, in Section 3 we give the main result (see Theorem \ref{Esistenza}). Finally, the fractional Laplacian case is studied in the last section. A concrete example of an application is presented in Example \ref{esempio0}.

\section{Preliminaries}\label{section2}

In this subsection we briefly recall the definition of the functional space~$X_0$, first introduced in \cite{sv, svmountain}.

Let $K:\erre^n\setminus\{0\}\rightarrow(0,+\infty)$ be a function with the properties that:
\begin{itemize}
\item[$(\rm k_1)$] $\gamma K\in L^1(\erre^n)$, \textit{where} $\gamma(x):=\min \{|x|^2, 1\}$;
\item[$(\rm k_2)$] \textit{there exists} $\beta>0$
\textit{such that} $$K(x)\geq \beta |x|^{-(n+2s)},$$
{\textit{for any}} $x\in \erre^n \setminus\{0\}$;
\item[$(\rm k_3)$] $K(x)=K(-x)$, \textit{for any} $x\in \erre^n \setminus\{0\}$.
\end{itemize}

 The functional space $X$ denotes the linear space of Lebesgue
measurable functions from $\RR^n$ to $\RR$ such that the restriction
to $\Omega$ of any function $g$ in $X$ belongs to $L^2(\Omega)$ and
$$
((x,y)\mapsto (g(x)-g(y))\sqrt{K(x-y)})\in
L^2\big((\RR^n\times\RR^n) \setminus ({\mathcal C}\Omega\times
{\mathcal C}\Omega), dxdy\big),$$ where ${\mathcal C}\Omega:=\RR^n
\setminus\Omega$. We denote by $X_0$ the following linear
subspace of $X$
$$X_0:=\big\{g\in X : g=0\,\, \mbox{a.e. in}\,\,
\RR^n\setminus
\Omega\big\}.$$
\indent We remark that $X$ and $X_0$ are non-empty, since $C^2_0 (\Omega)\subseteq X_0$ by \cite[Lemma~11]{sv}.
Moreover, the space $X$ is endowed with the norm defined as
$$
\|g\|_X:=\|g\|_{L^2(\Omega)}+\Big(\int_Q |g(x)-g(y)|^2K(x-y)dxdy\Big)^{1/2}\,,
$$
where $Q:=(\RR^n\times\RR^n)\setminus \mathcal O$ and
${\mathcal{O}}:=({\mathcal{C}}\Omega)\times({\mathcal{C}}\Omega)\subset\RR^n\times\RR^n$. It is easily seen that $\|\cdot\|_X$ is a norm on $X$; see, for
instance, \cite{svmountain}.\par
\indent By \cite[Lemmas~6 and 7]{svmountain} we can take in the sequel the function
\begin{equation}\label{normaX0}
X_0\ni v\mapsto \|v\|_{X_0}:=\left(\int_Q|v(x)-v(y)|^2K(x-y)dxdy\right)^{1/2}
\end{equation}
as norm on $X_0$. Also $\left(X_0, \|\cdot\|_{X_0}\right)$ is a Hilbert space with scalar product
$$
\langle u,v\rangle_{X_0}:=\int_Q
\big( u(x)-u(y)\big) \big( v(x)-v(y)\big)\,K(x-y)dxdy,
$$
see \cite[Lemma~7]{svmountain}.\par
 Note that in (\ref{normaX0}) (and in the related scalar product) the integral can be extended to all $\RR^n\times\RR^n$, since $v\in X_0$ (and so $v=0$ a.e. in $\RR^n\setminus \Omega$).\par
\indent While for a general kernel~$K$ satisfying
conditions from $(\rm k_1)$ to $(\rm k_3)$ we have that
$X_0\subset H^s(\RR^n)$, in the model case $K(x):=|x|^{-(n+2s)}$ the
space $X_0$ consists of all the functions of the usual fractional
Sobolev space $H^s(\RR^n)$ which vanish a.e. outside $\Omega$; see
\cite[Lemma~7]{servadeivaldinociBN}.\par
 Here $H^s(\RR^n)$ denotes
the usual fractional Sobolev space endowed with the norm (the
so-called \emph{Gagliardo norm})
$$
\|g\|_{H^s(\RR^n)}=\|g\|_{L^2(\RR^n)}+
\Big(\int_{\RR^n\times\RR^n}\frac{\,\,\,|g(x)-g(y)|^2}{|x-y|^{n+2s}}\,dxdy\Big)^{1/2}.
$$

\indent Before concluding this subsection, we recall the embedding
properties of~$X_0$ into the usual Lebesgue spaces; see
\cite[Lemma~8]{svmountain}. The embedding $j:X_0\hookrightarrow
L^{\nu}(\RR^n)$ is continuous for any $\nu\in [1,2^*]$, while it is
compact whenever $\nu\in [1,2^*)$, where $2^*:=2n/(n-2s)$ denotes the \textit{fractional critical Sobolev exponent}.\par

\indent For further details on the fractional Sobolev spaces we refer
to~\cite{valpal} and to the references therein, while for other
details on $X$ and $X_0$ we refer to \cite{sv}, where these
functional spaces were introduced, and also to \cite{sY, svmountain,
svlinking, servadeivaldinociBN}, where various properties of these spaces were
proved.\par
Finally, our abstract tool for proving the main result of the present paper is the following local minimum result due to Ricceri (see \cite{R4} and \cite{ricceri}).

\begin{theorem}\label{CV}
 Let $(E,\|\cdot\|)$ be a reflexive real Banach space and let $\Phi,\Psi:X\rightarrow \erre$ be two sequentially weakly lower semicontinuous functionals, with $\Psi$ coercive and $\Phi(0_E)=\Psi(0_E)=0$. Further, set $$J_\mu:=\mu\Psi+\Phi.$$ Then, for each $\sigma>\displaystyle\inf_{u\in X}\Psi(u)$ and each $\mu$ satisfying
 $$
\mu>-\frac{\displaystyle\inf_{u\in \Psi^{-1}((-\infty,\sigma])}\Phi(u)}{\sigma}
 $$
 the restriction of $J_\mu$ to $\Psi^{-1}((-\infty,\sigma))$ has a global minimum.
 \end{theorem}
\indent We cite the monograph \cite{k2} for related topics on variational methods adopted in this paper and \cite{c1,c2,c3} for recent nice results in the fractional setting.
\section{The Main Result}

Denote by $\mathcal{A}$ the class of all continuous functions $f:\erre\rightarrow \erre$ such that
$$
\displaystyle\sup_{t\in\erre}\frac{|f(t)|}{1+|t|^{\gamma-1}}<+\infty,
$$
for some $\gamma\in [1,2^{*})$. Further, if $f\in \mathcal{A}$ we put
$$
F(\xi):=\displaystyle\int_0^{\xi}f(t)\,dt,
$$
for every $\xi\in\erre$.\par
Let $h\in L^{\infty}(\Omega)$ and $f\in \mathcal{A}$. Consider the fractional problem
\begin{equation} \tag{$P_{K}^{h,f}$} \label{Alf}
\small\left\{
\begin{array}{ll}
-\mathcal L_Ku
= h(x)f(u) & \rm in \quad \Omega\\
 u=0 & {\mbox{in }} \erre^n\setminus \Omega.\\
\end{array}
\right.
\end{equation}
 We recall that a \textit{weak solution} of problem $(P_{K}^{h,f})$ is a function $u\in X_0$ such that
$$
\begin{array}{l} {\displaystyle \int_Q \big(u(x)-u(y)\big)\big(\varphi(x)-\varphi(y)\big)K(x-y)dxdy}\\
\displaystyle \qquad\qquad\qquad\qquad=\int_\Omega h(x)f(u(x))\varphi(x)dx,\,\,\,\,\,\,\,
\end{array}
$$
for every $\varphi \in X_0$.\par
 We observe that
problem~$(P_{K}^{h,f})$ has a variational structure, indeed it is the Euler-Lagrange equation of the functional $J_K:X_0\to \RR$ defined as follows
$$
J_K(u):=\frac 1 2 \|u\|_{X_0}^2 -\displaystyle\int_\Omega h(x)F(u(x))dx.
$$
\indent Note that the functional $J_K$ is Fr\'echet differentiable in $u\in X_0$ and one has
$$
\langle J'_K(u), \varphi\rangle = {\displaystyle \int_Q \big(u(x)-u(y)\big)\big(\varphi(x)-\varphi(y)\big)K(x-y)dxdy}
$$
$$
\qquad \qquad \qquad \qquad\, -\int_\Omega h(x)f(u(x))\varphi(x)dx,
$$
for every $\varphi \in X_0$.\par
 Thus, critical points of $J_K$ are solutions to problem~$(P_{K}^{h,f})$.
In order to find these critical points, we will make use of Theorem \ref{CV}.\par
\smallskip
\textbf{Notations}: Let $0\leq a<b\leq +\infty$. If $\lambda\in [a,b]$ and $\phi,\psi:\erre\rightarrow \erre$ are two assigned functions, we set
$$
g_\lambda^{\varphi,\psi}:=\lambda \psi-\varphi.
$$
\indent Denote
$$
M(\phi,\psi,\lambda):=
\left\{\begin{array}{ll}
{\rm the\, set\, of\, global\, minima\, of\,} g_\lambda^{\varphi,\psi} & \mbox{if}\,\,\,\lambda<+\infty\\
\emptyset & \mbox{if}\,\,\,\lambda=+\infty,
\end{array}\right.
$$
$$
\alpha(\phi,\psi,b):=\max\left\{\inf_{s\in \erre}\psi(s),\sup_{s\in M(\phi,\psi,b)}\psi(s)\right\},
$$
and
$$
\beta(\phi,\psi,a):=\min\left\{\sup_{s\in \erre}\psi(s),\inf_{s\in M(\phi,\psi,b)}\psi(s)\right\},
$$
adopting the conventions $\sup \emptyset=-\infty$ and $\inf \emptyset=+\infty$.\par
 Further, if $q\in [1,2^{*})$, we denote by $\mathfrak{F}_q$ the family of all the lower semicontinuous functions $\psi:\erre\rightarrow \erre$ such that:
\begin{itemize}
\item[$({\rm i}_1)$]$\displaystyle\sup_{s\in \erre}\psi(s)>0$;
\item[$({\rm i}_2)$]$\displaystyle\inf_{s\in \erre}\frac{\psi(s)}{1+|s|^q}>-\infty$;
\item[$({\rm i}_3)$]$\displaystyle\gamma_\psi:=\sup_{s\in \erre\setminus\{0\}}\frac{\psi(s)}{|s|^q}<+\infty$.
\end{itemize}

The next result, that can be viewed as a special case of \cite[Theorem 1]{R3}, will be crucial in the proof of the main theorem.

\begin{proposition}\label{Lemma}
Let $\phi,\psi:\erre\rightarrow \erre$ be two functions such that, for each $\lambda\in (a,b)$, the function $g_\lambda^{\varphi,\psi}$ is lower semicontinuous, coercive and has a global minimum in $\erre$. Assume that
$$
\alpha(\phi,\psi,b)<\beta(\phi,\psi,a).
$$
Then, for each
$$
r\in(\alpha(F,\psi,b),\beta(F,\psi,a)),
$$
there exists $\lambda_r\in(a,b)$, such that the unique global minimum of $g_{\lambda_r}^{\varphi,\psi}$ lies in $\psi^{-1}(r)$.
\end{proposition}

Set
$$
c_q:=\sup_{u\in X_0\setminus\{0_{X_0}\}}\frac{\|u\|_{L^q(\Omega)}^{q}}{\|u\|_{X_0}^{q}}.
$$
Note that, since $X_0\hookrightarrow L^q(\Omega)$ continuously, clearly $c_q<+\infty$.\par
\smallskip
With the above notations our result reads as follows.

\begin{theorem}\label{Esistenza}
Let $s\in (0,1)$, $n>2s$ and let $\Omega$ be an open bounded set of
$\RR^n$ with Lipschitz boundary and
$K:\RR^n\setminus\{0\}\rightarrow(0,+\infty)$ be a map
satisfying $(\rm k_1)$--$(\rm k_3)$.
Moreover, let $f\in \mathcal{A}$ and $h\in L^{\infty}_+(\Omega)\setminus\{0\}$. Assume that there exists $\psi\in \mathfrak{F}_q$ such that, for each $\lambda\in (a,b),$ the function $g_\lambda^{F,\psi}$ is coercive and has a unique global minimum in $\erre$. Further, suppose that there exists a number $r>0$ satisfying
$$
r\in(\alpha(F,\psi,b),\beta(F,\psi,a))
$$
and
\begin{equation}\label{F}
\sup_{\xi\in \psi^{-1}(r)}F(\xi)<\displaystyle\frac{r^{{2}/{q}}}{2(c_q\gamma_\psi \displaystyle\esssup_{x\in \Omega}h(x))^{{2}/{q}}\|h\|_{L^1(\Omega)}^{{(q-2)}/{q}}}.
\end{equation}
Then, problem $(P_{K}^{h,f})$ admits at least one weak solution which is a local minimum of the energy functional $J_K$ and
satisfies
$$
\int_Q|v(x)-v(y)|^2K(x-y)dxdy<\left(\frac{r\displaystyle \|h\|_{L^1(\Omega)}}{c_q\gamma_\psi\displaystyle\esssup_{x\in \Omega}h(x)}\right)^{2/q}.
$$
\end{theorem}
\begin{proof}
Let us apply Theorem \ref{CV} by choosing $E:=X_0$, and
$$
 \Phi(u):=-\int_\Omega h(x)F(u(x))\,dx,\qquad \Psi(u):=\displaystyle{\|u\|_{X_0}^2}
$$
for every $u\in E$.\par
 Set
\begin{equation}\label{g1}
\sigma:=\left(\frac{r\displaystyle \|h\|_{L^1(\Omega)}}{c_q\gamma_\psi\displaystyle\esssup_{x\in \Omega}h(x)}\right)^{2/q}.
\end{equation}

\indent We claim that
\begin{equation}\label{6}
\Psi^{-1}((-\infty,\sigma])\subseteq D,
\end{equation}
where
$$
D:=\left\{u\in L^{q}(\Omega):\int_\Omega h(x)\psi(u(x))dx\leq r \|h\|_{L^1(\Omega)}\right\}.
$$

Indeed, since $X_0\hookrightarrow L^q(\Omega)$, it follows that
\begin{equation}\label{4}
\Psi^{-1}((-\infty,\sigma])\subseteq\left\{u\in L^{q}(\Omega):\|u\|_{L^q(\Omega)}\leq c_q^{1/q}\sqrt{\sigma}\right\}.
\end{equation}

On the other hand, taking into account that $\psi\in \mathfrak{F}_q$, one also has
\begin{equation}\label{5}
\int_\Omega h(x)\psi(u(x))dx\leq\gamma_\psi\|u\|_{L^q(\Omega)}^q\esssup_{x\in \Omega}h(x).
\end{equation}
Hence, inclusion \eqref{6} follows from inequalities \eqref{4} and \eqref{5}.

\indent Now, for each parameter $\lambda\in (a,b)$ denote by $\xi^{\star}_\lambda$ the unique global minimum (in $\erre$) of the real function $g_\lambda^{F, \psi}$.\par
 By Lemma \ref{Lemma}, since by assumption
$$
r\in (\alpha(F,\psi,b), \beta(F,\psi,a)),
$$
there exists ${\lambda_r}\in (a,b)$ such that $\psi(\xi^{\star}_{\lambda_r})=r$.\par
\indent Hence, since
\begin{equation}\label{gu}
g_{\lambda_r}^{F,\psi}(\xi^{\star}_{\lambda_r})\leq g_{\lambda_r}^{F,\psi}(\xi),
\end{equation}
 for every $\xi\in \erre$, it follows that
 \begin{equation}\label{7}
 F(\xi^{\star}_{\lambda_r})=\sup_{\xi\in \psi^{-1}(r)}F(\xi).
 \end{equation}

 Bearing in mind that $h$ is non-negative, by \eqref{gu} one has
 \begin{equation}\label{gu2}
g_\lambda^{F,\psi}(\xi^{\star}_{\lambda_r})h(x)\leq h(x)g_{\lambda_r}^{F,\psi}(\xi),
\end{equation}
for a.e. $x\in \Omega$. Hence, inequality \eqref{gu2} implies that
 \begin{equation}\label{gu3}
g_{\lambda_r}^{F,\psi}(\xi^{\star}_{\lambda_r})\|h\|_{L^1(\Omega)}\leq \int_\Omega h(x)g_{\lambda_r}^{F,\psi}(u(x))dx,
\end{equation}
for every $u\in L^q(\Omega)$.\par
\indent Exploiting \eqref{gu3}, for every $u\in D$, one has
$$
\int_\Omega h(x)F(u(x))dx\leq F(\xi^\star_{\lambda_r})\|h\|_{L^1(\Omega)}.
$$
Owing to \eqref{7}, the above inequality assumes the form
 \begin{equation}\label{colsup}
\int_\Omega h(x)F(u(x))dx\leq \sup_{\xi\in \psi^{-1}(r)}F(\xi)\|h\|_{L^1(\Omega)},
\end{equation}
for every $u\in D$.\par
Observing that
$$
r=\frac{\displaystyle c_q\gamma_\psi \sigma^{q/2}\esssup_{x\in \Omega}h(x)}{\|h\|_{L^1(\Omega)}},
$$
and since inclusion \eqref{6} holds, it follows that
\begin{equation}\label{oggi}
\displaystyle\sup_{u\in \Psi^{-1}((-\infty,\sigma])}\int_\Omega h(x)F(u(x))\,dx\leq \sup_{\xi\in \psi^{-1}(r)}F(\xi)\|h\|_{L^1(\Omega)}.
\end{equation}

\indent Finally, relations \eqref{F} and \eqref{oggi} yield
$$
\frac{\displaystyle\sup_{u\in \Psi^{-1}((-\infty,\sigma])}\int_\Omega h(x)F(u(x))\,dx}{\sigma}<\frac{1}{2},
$$
that is
$$
\frac{1}{2}>-\frac{\displaystyle\inf_{u\in \Psi^{-1}((-\infty,\sigma])}\Phi(u)}{\sigma}.
$$
 \indent
    Then, the assertion of Theorem \ref{CV} follows and the existence of one weak solution $u\in \Psi^{-1}((-\infty,\sigma))$ to our problem is established.
\end{proof}
\begin{remark}\rm{
The above existence theorem extends to the nonlocal setting some results, already known in the literature in the case of the classical $p$-Laplace operator (see \cite{R2}).}
\end{remark}
\section{The fractional Laplacian case}
As observed in Section 2, by taking $K(x):=|x|^{-(n+2s)}$, the
space $X_0$ consists of all the functions of the usual fractional
Sobolev space $H^s(\RR^n)$ which vanish almost everywhere outside $\Omega$; see
\cite[Lemma~7]{servadeivaldinociBN}.\par
 In this case $\mathcal L_K$ is the fractional Laplace operator defined as
$$
-(-\Delta)^s u(x):=
\int_{\erre^n}\frac{u(x+y)+u(x-y)-2u(x)}{|y|^{n+2s}}\,dy,
\,\,\,\,\, x\in \erre^n.
$$
\indent By \cite[Proposition~9 and Appendix~A]{svlinking}, we a variational characterization of the first eigenvalue (denoted by ${\lambda_{1,s}}$) of the problem
$$ \left\{
\begin{array}{ll}
(-\Delta)^s u=\lambda u & {\mbox{ in }} \Omega\\
u=0 & {\mbox{ in }} \erre^n\setminus \Omega,
\end{array} \right.
$$
as follows
\begin{equation}
\lambda_{1,s}=\min_{u\in X_0\setminus{\left\{0_{X_0}\right\}}}
\frac{\displaystyle\int_{\mathbb{R}^{2n}}\frac{\left|u(x)-u(y)\right|^{2}}{|x-y|^{n+2s}}dx\,dy}
{\displaystyle \int_{\Omega}u(x)^{2}dx}.\label{lambda1}
\end{equation}
\indent In the sequel it will be useful the following regularity result for the eigenvalues
of $(-\Delta)^s$ proved in \cite[Theorem 1]{dueoperatori}. See also \cite[Proposition 2.4]{sY} for related topics.
\begin{proposition}\label{propossizioneRaff}
Let $e\in X_0$ and $\lambda>0$ be such that
$$
\langle e,\varphi\rangle_{X_0}=\lambda \int_\Omega e(x)\varphi(x)dx,
$$
for every $\varphi\in X_0$. Then $e\in C^{0,\alpha}(\bar\Omega),$ for some $\alpha\in (0,1),$ i.e. the function $e$ is H\"{o}lder continuous up to the boundary.
\end{proposition}
\indent Taking into account the above facts, a meaningful consequence of Theorem \ref{Esistenza} is the following one.
\begin{theorem}\label{Esistenza2}
Let $s\in (0,1),$ $n>2s$ and $\Omega$ be an open bounded set of
$\RR^n$ with Lipschitz boundary.
Moreover, let $h\in L^{\infty}(\Omega)\setminus\{0\}$ with $\displaystyle\essinf_{x\in \Omega}h(x)>0$  and let $f:[0,+\infty)\rightarrow [0,+\infty)$ be a continuous function satisfying the following hypotheses$:$

\begin{itemize}
\item[$({\rm h}_1)$] for some $q\in\left[1,2^*\right)$ the function
$$
t\mapsto \frac{f(t)}{t^{q-1}}
$$
is strictly decreasing in $(0,+\infty)$ and $\displaystyle\lim_{t\rightarrow +\infty}\frac{f(t)}{t^{q-1}}=0$$;$
\item[$({\rm h}_2)$] one has
$$
\liminf_{\xi\rightarrow 0^+}\frac{F(\xi)}{\xi^2}>\frac{\lambda_{1,s}}{2\displaystyle\essinf_{x\in\Omega}h(x)};
$$
\item[$({\rm h}_3)$] there exists $\xi_0>0$ such that
$$
F(\xi_0)<\displaystyle\frac{\xi_0^{{2}}}{2(c_q\displaystyle\esssup_{x\in \Omega}h(x))^{{2}/{q}}\|h\|_{L^1(\Omega)}^{{(q-2)}/{q}}}.
$$
\end{itemize}
Then, the problem

$$
\begin{array}{l} {\displaystyle \int_{\erre^{2n}}
\frac{(u(x)-u(y))(\varphi(x)-\varphi(y))}{|x-y|^{n+2s}}dxdy}\\
\displaystyle \qquad\qquad\qquad\qquad\qquad=\int_\Omega h(x)f(u(x))\varphi(x)dx,\,\,\,\,\,\,\,
\end{array}
$$
for every
$$
\varphi\in H^{s}(\erre^n)\,\, {\rm such\,\, that}\,\, \varphi= 0\,\,\, {\rm a.e.\,\, in}\,\, \erre^n\setminus\Omega,
$$
admits at least one non-negative and non-zero weak solution
$u\in  H^s(\erre^n),$
such that $u=0$ a.e. in $\erre^n\setminus\Omega$. Moreover,
$u$ is a local minimum of the energy functional $J_K$ and
satisfies
\begin{equation}\label{F2}
\int_{\RR^{2n}}\frac{|u_j(x)-u_j(y)|^2}{|x-y|^{n+2s}}\,dxdy<\left(\frac{\displaystyle \|h\|_{L^1(\Omega)}}{c_q\displaystyle\esssup_{x\in \Omega}h(x)}\right)^{2/q}\xi_0^2.
\end{equation}
\end{theorem}
\begin{proof} Let us define
$$
\widetilde{f}(t):=
\left\{\begin{array}{ll}
f(t) & \mbox{if}\,\,\,t\geq 0\\
f(0) & \mbox{if}\,\,\,t<0,
\end{array}\right.
$$
and consider the following problem
\begin{equation}  \tag{$D_{\widetilde{f}}^{h}$} \label{GiornoV}
\left\{
\begin{array}{ll}
(-\Delta)^s u=h(x)\widetilde{f}(u) & {\mbox{\,\,\rm in }} \Omega\\
u=0 & {\mbox{ \rm in }} \erre^n\setminus \Omega.
\end{array} \right.
\end{equation}
\indent By \cite[Lemma 6]{svw} every weak solution of problem \eqref{GiornoV} is non-negative in $\Omega$. Furthermore, every non-negative solution of problem \eqref{GiornoV} also solves our initial problem. Taking $a:=0$  and $b:=+\infty,$ and by exploiting Theorem \ref{CV} with $$\psi(t):=|t|^q,\,\,\forall\, t\in \erre$$
substantially arguing as in \cite{R2}, the existence of one weak solution of problem \eqref{GiornoV} which is a local minimum of the associated energy functional (namely $\widetilde{J}_K$) satisfying \eqref{F2} is established.\par
 In conclusion, we prove that $0_{X_0}$ is not a local minimum of $\widetilde{J}_K$, i.e. the obtained solution is non-zero.\par
 For this purpose, let us observe that the first eigenfunction $e_1\in X_0$ is positive in $\Omega$, see \cite[Corollary 8]{svw}, and it follows by (\ref{lambda1}) that
\begin{equation}\label{ScV}
\|e_1\|^2_{X_0}={\lambda_1}\int_\Omega e_1(x)^2 dx.
\end{equation}
\indent Thanks to $({\rm h}_2)$, there exists $\delta>0$ such that
$$
F(\xi)>\frac{\lambda_{1,s}}{2\displaystyle\essinf_{x\in\Omega}h(x)}\xi^2,
$$
\noindent for every $\xi\in (0,\delta)$.\par
 Now, by Proposition \ref{propossizioneRaff}, one has that $e_1\in C^{0,\alpha}(\bar\Omega)$. Hence, we can define $\theta_{\eta}(x):=\eta e_1(x)$, for every $x\in \Omega$, where $$\eta\in \Lambda_\delta:=\left(0,\displaystyle\frac{\delta}{\displaystyle\max_{x\in \Omega}e_1(x)}\right).$$
\indent Taking into account \eqref{ScV}, we easily get
 \begin{eqnarray*}
\int_\Omega h(x)F(\theta_\eta(x))dx &> & \frac{\lambda_1\displaystyle\int_\Omega h(x)\theta_\eta(x)^2dx}{2\displaystyle\essinf_{x\in\Omega}h(x)}\\\nonumber
               &\geq& {\frac{\lambda_1}{2}\displaystyle\int_\Omega \theta_\eta(x)^2dx}\\\nonumber
               &=& \frac{1}{2}\|\theta_\eta\|^2_{X_0},\nonumber
\end{eqnarray*}
\noindent that is
$$
\widetilde{J}_K(\theta_\eta)=\frac{1}{2}\|\theta_\eta\|^2_{X_0}-\int_\Omega h(x)F(\theta_\eta(x))dx<0,
$$
\indent for every $\eta\in \Lambda_\delta$.\par
\noindent The proof is thus complete.
\end{proof}

\indent It is easy to see that Theorem \ref{Esistenza3} in Introduction is a consequence of Theorem \ref{Esistenza2}. A direct application of this result reads as follows.

\begin{es}\label{esempio0}\rm{
Let $s\in (0,1)$, $n>2s$ and let $\Omega$ be an open bounded set of $\erre^n$
with
Lipschitz boundary. By virtue of Theorem \ref{Esistenza2}, for every
$
\alpha>\lambda_{1,s},
$
the following nonlocal problem
$$ \left\{
\begin{array}{ll}
(-\Delta)^s u=\displaystyle\frac{\alpha u}{1+u^2} & {\mbox{ \rm in }} \Omega\\
u=0 & {\mbox{ \rm in }} \erre^n\setminus \Omega,
\end{array} \right.
$$
admits at least one non-negative and non-zero weak solution
$u_\alpha\in  H^s(\erre^n),$
such that $u_{\alpha}=0$ a.e. in $\erre^n\setminus\Omega$.}
\end{es}

\begin{remark}\rm{
We observe that a special case of our results ensures that if $f:[0,+\infty)\rightarrow [0,+\infty)$ is any positive $C^1$-function such that $f(0)=0$, $f(t)/t$ is strictly decreasing in $(0,+\infty)$, $f(t)/t\rightarrow 0$ as $t\rightarrow +\infty$ and $f'(0)>\lambda_1$, then the following problem
$$ \left\{
\begin{array}{ll}
(-\Delta)^s u=f(u) & {\mbox{ \rm in }} \Omega\\
u=0 & {\mbox{ \rm in }} \erre^n\setminus \Omega,
\end{array} \right.
$$
admits at least one non-negative and non-zero weak solution
$u\in  H^s(\erre^n),$
such that $u=0$ a.e. in $\erre^n\setminus\Omega$.}
\end{remark}

{\bf Acknowledgements.}
  The manuscript was realized within the auspices of the GNAMPA Project 2014 titled {\it Propriet\`{a} geometriche e analitiche per problemi non-locali} and the SRA grants P1-0292-0101 and J1-5435-0101.

\end{document}